\documentclass[preprint,12pt,authoryear]{elsarticle}%final,5p,times,twocolumn,
\usepackage[T2A]{fontenc}
\input cyracc.def
\usepackage{bm}%eulerpx,bm}
\usepackage{amsmath}
\usepackage{amssymb}

\usepackage[main=english,russian]{babel}
\usepackage{comment}
\usepackage{lineno}
\usepackage{subcaption}
\graphicspath{{./Figures/},{./Pictures/},{./plots/},{./Bio/}}
\usepackage[colorlinks=true]{hyperref}%,backref=true
  \hypersetup{
    pdftitle={Sequential selections with minimization of failure}, %%<--To wymieni? 
    pdfauthor={Krzysztof J. Szajowski}, %%<--TO wymieni?
    colorlinks,
    urlcolor=blue,
    filecolor=magenta,
    citecolor=green, 
    linkbordercolor={1 1 1}, % set to white
    citebordercolor={1 1 1},  % set to white
    urlbordercolor={ 1 1 1}  % set to white
  } 
  
 \RequirePackage[hyperpageref]{backref} 
    \renewcommand*{\backref}[1]{}  
    \renewcommand*{\backrefalt}[4]{
       \ifcase #1 
          No cited.
      \or
          Cited on p. #2.
       \else
          Cited on pp. #2.
       \fi}  
\def\bE{\bm{E}}
\def\bP{\bm{P}}
\def\bQ{\bm{Q}}
\def\bT{\bm{T}}

\def\BbbE{\mathbb E}
\def\BbbN{\mathbb N}%Natural numbers
\def\cB{{\mathcal B}}
\def\cF{\mathcal{F}}
\def\cS{\mathcal S}

\newcommand*{\MR}[1]{\href{http://www.ams.org/mathscinet-getitem?mr=#1&return=pdf}{\underline{MR #1}}}
\newcommand*{\PMC}[1]{\href{https://www.ncbi.nlm.nih.gov/pmc/articles/#1/}{PubMed Central PMCID: #1 }} %PMC4737670.
\newcommand*{\PMID}[1]{\href{https://pubmed.ncbi.nlm.nih.gov/#1/}{PMID: #1 }} %PMC4737670. 

\newcommand*{\ZBL}[1]{\href{http://www.zentralblatt-math.org/zmath/en/advanced/?q=an:#1&format=complete}{Zbl #1}}

\newcommand*{\doi}[1]{\href{http://dx.doi.org/#1}{doi: #1}}

\newcommand*{\KSzdoi}[1]{\href{http://dx.doi.org/#1}{doi: #1}}
\newcommand{\abbreviations}[1]{%
\vspace{12pt}\noindent{\selectfont\textbf{Abbreviations}\par\vspace{6pt}\noindent {\fontsize{9}{9}\selectfont #1}\par}}
\usepackage{soul,color}
\newtheorem{nthm}{Theorem}[section]
\newtheorem{nlem}[nthm]{Lemma}
\newtheorem{ncor}[nthm]{Corollary}
\newdefinition{nrem}{Remark}[section]
\newdefinition{ndef}[nthm]{Definition}

\newproof{proof}{Proof}
\newproof{pot}{Proof of Theorem \ref{thm2}}
\newproof{notabene}{\color{blue}{Notabene}}

\begin{document}

\begin{frontmatter}



\title{Sequential selections with minimization of failure }

\author[inst1]{Krzysztof J. Szajowski}
\ead{Krzysztof.Szajowski@pwr.edu.pl}
\affiliation[inst1]{organization={Faculty of Pure and Appl. Math.},%Department and Organization
                     addressline={Wroclaw University of Science and Technology\\
						                      Wybrzeze Wyspianskiego 27}, 
                            city={Wroclaw},
                        postcode={50-370}, 
                           state={Dolnoslaskie},
                         country={Poland}
										}

%\author[inst2]{Author Two}
%\author[inst1,inst2]{Author Three}

%\affiliation[inst2]{organization={Department Two},%Department and Organization
%            addressline={Address Two}, 
%            city={City Two},
%            postcode={22222}, 
%            state={State Two},
%            country={Country Two}}

\begin{abstract}
The decision-maker (\textbf{DM}) sequentially evaluates up to $ N $ of different, rankable options. \textbf{DM} must select exactly the best one at the moment of its appearance. In the process of searching, \textbf{DM} finds out with each applicant whether she is the best applicant among those assessed so far (we call him a candidate). \textbf{DM} cannot return to rejected candidates. We discuss the psychological aspects of this selection problem, known in the literature as the secretary problem. The analysis is based on knowledge of the chances, and a subjective assessment of acceptance of the positive and negative effects \textbf{DM}'s decision. The acceptance assessment of success and failure is presented on subjective scales.  We set an optimal policy that recommends analyzing applicants up to a certain point in time (a threshold time) without selecting any of them and then selecting the next encountered candidate. The determined optimal threshold depends on the level of acceptance of the positive and negative effects of the choice.  This issue is discussed in the article.
%\st{We discuss an extension of the secretary problem in which the decision maker (\textbf{DM}) sequentially observes up to $N$ applicants whose values are different and rank-able. The \textbf{DM} must select exactly one applicant, cannot recall released applicants, and receives a payoff of $\alpha$ choosing the best one, or he is penalized of $-\beta$ or $-\gamma$ when he makes other choices or nothing, respectively. For each encountered applicant, the \textbf{DM} only learns whether the applicant is the best so far. The optimal policy is constructed, and it dictates skipping the first $k^\star-1$ applicants, and then selecting the next encountered applicant whose value is better than all previously interviewed. There is no guarantee that the optimal strategy for these problems is unique. This ambiguity in the solution is particularly interesting when we analyze the time spent on recruitment. The issue is discussed in the paper.} 
\end{abstract}

\begin{comment}
\begin{graphicalabstract}
\includegraphics{gr3DAxes05p95a.eps}
\end{graphicalabstract}
%\end{document}
\begin{highlights}%\itemsep1pt \parskip2pt \parsep1pt
%\begin{itemize}
\item The reason for accepting candidates earlier may be subconsciously taking into account the penalty of a wrong choice or a lack of choice.
\item An additional factor that regulates the selector's behavior is the shortening of the time to complete the selection procedure.
\item Unavailable candidates additionally increase the chance and speed up attempts to select candidates.
\item Parameterization of the level of impact of individual risks increases the possibility of adjusting the strategy to the perception of the selector. 
%\end{itemize}
\hfill
\end{highlights}
\end{comment}


\end{frontmatter}

%%\linenumbers
%%\modulolinenumbers[2]
%%\renewcommand\thelinenumber{\color{red}\arabic{linenumber}}
%\vspace{-3ex}


%%       \citet{<label>} ==> Jones et al. (2015)
%%       \citep{<label>} ==> (Jones et al., 2015)

\section{\label{sec:sample1}Introduction.} The author's interest in this work are the mathematical models of decision-making under uncertainty. %{\color{blue}
The considered model supports several theories in psychology. The aim of the study is to analyze the avoidance--striving psychological conflict (v. \cite[Chapt. 13]{Sok2005:PsychDecRyz}, \cite{Koz1981:PDT}) in the selection process. The decision-maker (we then use the acronym \textbf{DM} to recall the decision-maker) tries to avoid risks and undesirable situations, and at the same time, he strives to achieve the desired, assumed goals. Adopting this approach leads to the search for a selection model based on the assumption that the final effect has two general states. One is equated with loss, danger, and worst results. The second is usually equated with the achievement of the assumed goal, which brings specific benefits. The choice is a strive-avoidance conflict. It expresses the split between "fear and greed" (v. \cite{Coo1975:Portfolio} introducing \cite{Mar1959:Portfolio} to psychology, cf. also \cite{CapKop2014:Portfolio}). Consider the model of sequential analysis of options in order to choose the best option that they are all different - so only one is the best and the choice of this one will satisfy him if all of them are not exhausted. The choice is risky because the options come in random order. However, it is easy to assess the probability that the analyzed object is the best. The results of the research on this model are known in the literature, both theoretical and experimental. In experimental research, attempts were made to assess to what extent the decision-makers correctly assess the chances that the tested object is the best. This was related to the expansion of information about the analyzed objects, which is justified by the real problems to which this type of model is applied, but does not answer the question about the role of motivational mechanisms, including personality-conditioned goals in the decision problem, in which the effect of the choice is to satisfy goals or failure. Here, how the way decisions are made is influenced by the conflict between "fear and hope". Such selection tasks can be found in the transplant profession. To put it simply, when an organ for transplantation appears, the pending patients are checked for compliance. One of the most important factors for the success of the procedure is the selection of the best patient for the available just now material at the hospital. There is a risk of transplant organ rejection with no compatible patient selection. This is a situation we fear. As a result of such a "matching" procedure, if a compatible patient is not selected early enough, we will lose the opportunity to perform the procedure due to the lack of a candidate for the procedure - this is another circumstance we are concerned about. The candidate's material and the chance of a better one will be significantly lower than the chance that the analyzed patient is the best (and we do not know this at the time of the examination because the remaining patients are to be examined later) (v. \cite{Kras2010:organ}). 

The mathematical model of such a decision-making process, in various modifications, it is known as the \emph{secretary problem} or the  \emph{best choice problem} (\textbf{BCP})\label{KSzBCP}. In a classic task, the pursuit of the goal is clear - we are looking for a strategy that maximizes the probability of success. Now, we will also analyze our attitude to failure by assessing the usefulness of each possible end of the decision-making process. This approach will explain the lack of the observed compliance of the decisions made with the optimal ones, determined in the model with known probabilities. This is a different approach than in \cite{SeaRap1997:HeuristicBCP}, where such a discrepancy is explained by an incorrect assessment of opportunities by \textbf{DM}.%}

Every unit of time, observation of the implementation of a random variable appears. Its evaluation, together with previous observations, is a source of knowledge aimed at selecting the best implementation at the time of its appearance. A precise description of the objective function and the properties of the observation sequence will allow distinguishing those elements of the modeled phenomenon that affect the rationalization of the procedure - the determination of the optimal strategy. The analyzed phenomenon is common, and we have examples for which the introduced model is a mathematical description. The observation of actual behavior and choices has led  to the conclusion that the actions do not  derive from model suggestions. %\st{Although,  decision makers do not seek guidance in the conduct of solutions optimal for mathematical models, but rather try to repeat the strategies used in the past, an in-depth analysis of the models may prompt behaviorists to familiarize themselves with new observations.}
Although,  in solving every day problems, decision makers rarely do not seek guidance in mathematical models and the optimal or rational solutions delivered by them, but rather try to repeat the acceptable results obtained by strategies used in the past, an in-depth analysis of the mathematical models may prompt behaviorists to familiarize themselves with new observations concerning the psychological aspects of final consequence of the decision. 
  
The purpose of further considerations is to show that one can try to influence rational behavior by pointing to seemingly insignificant factors. It is worth considering positive and negative effects of a choice decision - because the strategy is the result of an attempt to strengthen the expected effect and weaken the negative unwanted effects. This is true even when in the decision-making process is achieving a precise goal. If this information is related to preparing for an exceptional decision, e.g., buying expensive goods or choosing a life partner, then we rationalize the assessment of the moment when in our opinion we are ready to take the final action (choosing the best condition). There are mathematical models related to such a task, and their study is of interest to both mathematicians and psychologists. Theoretical and experimental analyzes were developed. The theoretical analysis of the models requires clarifying the details, for a better understanding of the modeled phenomenon and greater consistency of assumptions with the cases encountered in practice. In seeking such an adjustment, the conditions in the adopted model can be weakened by the deviation from the actual natural decision-making process in relation to the suggested rational behavior achieved.  This approach to the problem led to a theoretical analysis of models with weakened assumptions in the work of \citeauthor{Bea2006:NewBCP} (\citeyear{Bea2006:NewBCP}), supported by experimental studies by \citeauthor{BeaMurRap2005:SPExp}~(\citeyear{BeaMurRap2005:SPExp}) and the theoretical model of the paper by \citeauthor{ChoMorRobSam1964:OSSP}~(\citeyear{ChoMorRobSam1964:OSSP}). To increase the usefulness of the \citeauthor{Bea2006:NewBCP}'s model, parameters can be introduced that take into account additional factors influencing the assessment of the situation in decision-makers (\citeauthor{Sza2009:Cardinal}~(\citeyear{Sza2009:Cardinal}), \citeauthor{FerKra2010:CardinalBCP}~(\citeyear{FerKra2010:CardinalBCP})). The classic, simplified selection model assumed that the individual observations do not have a specific, scalar value, the introduction of which in the Beardin's model allowed to explain the tendency to end the observation period earlier than in the classic model of ending the observation period and moving to making final decisions. The ability to take into account such aspects may also allow benefits from the analyses carried out in the work of \citeauthor{AngGut2020:When}~(\citeyear{AngGut2020:When}). 

\subsection{\label{KSzComBCP}A more complete problem of the best choice. }
The crucial assumptions of the considered decision process, called in the former consideration the best choice problem,  are following (v. \citeauthor{Gar60a}~(\citeyear{Gar60a,Gar60b}), \citeauthor{Dyn63}~(\citeyear{Dyn63}), \citeauthor{GilMos1966:maximum}~(\citeyear{GilMos1966:maximum}), \citeauthor{fre83:review}~(\citeyear{fre83:review}), \citeauthor{fer89:who}~(\citeyear{fer89:who}), \citeauthor{sam91:secpro}~(\citeyear{sam91:secpro}) and the monographs  \cite{BerGne1984:ZNV}):  
\begin{enumerate}\itemsep1pt \parskip1pt \parsep1pt
\item There is only one item needed.
\item The number of objects, N, is known in advance. 
\item They are interviewed sequentially in a random order.
\item All the items can be ranked from the best to the worst without any ties. Moreover, the  decision  to accept  or to reject  an item under review must  be  based  solely on  the  relative ranks of the viewed objects. 
\item There is no recall to rejected items later. 
\item \label{DMaim}DM is satisfied with nothing but the very best. In the case of choosing what he needs, the winning payoff is $1$, otherwise $0$.
\end{enumerate}
The proposed mathematical model is based on suitable sequences of random variables and strategies, which are stopping times related to the observed sequences of relative ranks. The optimally means \textit{the maximization of the chance} of realizing the plan of choosing the best option. It is well known that the optimal selection time has the threshold form: at the beginning, we observe and register the relative ranks of the investigated objects. After reaching $k^\star$ item, you choose the first Relative Best. The strategy is appealing, and its simplicity was emphasized by extending it to the odds' algorithm (cf. \citeauthor{Bru2000:Odds}~(\citeyear{Bru2000:Odds})).   

If the selector applies the theoretical solution, statistically the chance of success will be $ \frac{k^\star-1}{N} $. It does not include all the final effects of the application of this strategy, which are reduced to three mutually exclusive events:
\begin{enumerate}\itemsep1pt \parskip1pt \parsep1pt
\item[($\alpha$)] he is successful and chooses the best object;
\item[($\beta$)] he fails because the object he chooses is not the best\footnote{Perhaps it is worth considering the degree of failure resulting from the fact that even if we focus our attention on choosing the best object, we will accidentally choose an object with an absolute rank close to the best. The chances of this are high as shown by the results obtained by~\citeauthor{QuiLaw1996:Exact}~(\citeyear{QuiLaw1996:Exact}). It is choice of a simple majority of positive and a simple minority of negative aspects hire. };
\item[($\gamma$)] the decision maker will not succeed because the strategy used will not lead to the selection of an object.
\end{enumerate}
The described events are exclusive, their observation and interpretation after the completion of the selection is unambiguous. The selector can take into account each of these events to a varying degree, which can be differentiated by allocating proportional wins and penalties, denoted by $\alpha$, $-\beta$ and $-\gamma$, respectively, where $\alpha,\beta,\gamma\in [0,1]$\footnote{It is worth emphasizing that the end of the selection process cannot be described using the cost per observation (cf.~\citeauthor{BarGov1978:Cost}~(\citeyear{BarGov1978:Cost}), \citeauthor{Yeo1998:Costs}~(\citeyear{Yeo1998:Costs})) or discounting (v. \citeauthor{RasPli1975:Discount}~(\citeyear{RasPli1975:Discount}).}.
%\begin{notabene}\end{notabene}

We want to show that if the decision-maker focuses not only on the main goal, but is also not indifferent to him whether he makes a mistake in choosing the wrong object or its actions will not lead to a choice, he will still use a threshold strategy, but it will be different from the one he uses when these factors are not taken into account. %\st{ The extent to which the possible project endings are taken into account by the decision maker can be assessed on a scale:  $ 0 $ to $ 1 $. The higher its value, the more important this aspect for the decision maker is-- undesirable events will have a negative effect, and the desired ones will have a positive effect.}
We assume that a conscious decision-maker attaches different importance to the potential effects of his decision. We evaluate the significance of each of the possible scenarios for the decision-maker on a scale from $ 0$ to $ 1 $. The higher its value, the more important this aspect is for the decision-maker - adverse events will have a negative impact, and the desired - positive ones.

\subsection{The results.}
The aim of the presented research is to analytically show the impact of the above issue on the form of the optimal strategy (optimal stopping time) and the expected selection time. There is a close relationship between the profit-loss ratio in the problem of choosing the best item for the duration of the decision-making process. It is difficult to predict the exact relationship of these quantities without a detailed analysis. The research covers the problem of choosing the best object, taking into account the penalty for choosing the wrong observation and a different penalty for not choosing the right one. This problem is also addressed when the access to the desired observation is uncertain, but it is possible to retry subsequent observations.   

 The rest of the paper is structured as follows. Section~\ref{KSzVarRisk} provides the formal model of the best item selection, when the decision maker is motivated and conditioned by the awareness of the decision made. These considerations are supported by the method of the optimal stopping for discrete-time Markov processes with a finite horizon, which are shortly recalled in Section~\ref{KSzSecOptSP} of the Appendix. The solution has an analytical form which allows various applications. Section~\ref{KSzDurNIBCP} describes an analysis of the expected time of the selection process. We close with a few concluding remarks in Section \ref{KSzFinRem}. 

\section{\label{KSzNIBCP}Sequential decision problems.} In this section we present the main results of the work. They concern the best object selection model (cf. Introduction in part \ref{KSzComBCP}). We study optimal selection strategies in a model situation in which we analyze the attitude to risks. The technical side of the research comes down to the analysis of optimal stopping problems (v. Section~\ref{KSzSecOptSP} where we remind the formulation of this optimization task). Next, we will show the details of risk approach modeling and the analytical form of optimal strategies. We will discuss the impact of risk attitudes (the level of ”fear and hope”) on optimal strategies. In the last part of this section, we will conduct a similar analysis in the event that the effects of the selection strategy are uncertain - the observed state predicts our expectations but may not be available. This unavailability is an additional risk factor that we propose to include in the model. If the stream of available options is not exhausted, the search continues (v. \cite{Smi1975:uncertain}).

%%%%%%%
\subsection{\label{KSzVarRisk}The attitude towards risk in the problems of choosing the best object.}
One of the basic sequential selection models is the simple secretary problem described in Section \ref{KSzComBCP}\footnote{For wider references and history of the secretary problem the reader is reffered to philosophical review paper by \cite{fer89:who} and in the monograph by \cite{BerGne1984:ZNV}.}.%,BerGne1984:BCP}.} 
The decision maker interviews exactly $N$ options which  are  sequentially presented in random order. He can make one choice without recall to past decision stages. The \textbf{DM} is rewarded by $\alpha$ ($-\beta$), if he accepts an option which is (is not) absolutely best. Doing so does not always lead to the selection of one of the available options. Since the selector's task is to select an object with specific properties, if in the decision-making process it reaches the last possible option, and it is not desired, it means that the strategy used does not allow for the selection of a potential good solution. 

With the sequential observation of the applicants, we connect natural probability space $(\Omega,\cF,\bP)$. The elementary events $\omega  \in \Omega$ are the permutation $\ell_1,\ell_2,\ldots,\ell_N$ of all applicants, and the probability measure $\bP$ is the uniform distribution on $\Omega$. For the sake of attention, let's assume that $ \ell_j $ is the $j$-th applicant's label, and $ \ell_i \neq \ell_j $ when $ i \neq j $. Applicants apply the permutation order. Let $ Z_k = \# \{1 \leq j \leq N: \ell_j \leq \ell_k \} $\label{KSzRR1} be the absolute rank of the $k$-th applicant, and $ Y_k = \# \{1 \leq j \leq k: \ell_j \leq \ell_k \} $ its relative rank. The observable sequence of relative ranks $Y_k,\ k=1,2,\ldots,N$ defines the sequence of the $\sigma$-fields $\cF _k =\sigma (Y_1,\ldots,Y_k),\ k=1,2,\ldots,N$.  The random variables $Y_k$ are independent and $P(Y_k = i) = 1/k$ for $i\in\{1,2,\ldots,k\}$. Denote $\cS^N$ the set of all Markov times $\tau$ with respect to the $\sigma$-fields $\{\cF_k\}_{k=1}^N$ bounded by $N$. We want to forecast, on the basis of relative ranks, the moment of the appearance of the object with the lowest absolute rank. The success is granted by $\alpha$ and the failure is penalized by $\beta$. The case of no choice gives the penalty  $\gamma$. This assumption should be taken into account by extending the state space $\BbbE=\{1,2,\ldots,N\}\cup\{\partial\}$, where $\partial$ is the \emph{absorbing state} which concerns the situation when, while searching for the object with the lowest rank, we miss the moment of its appearance and no new candidate will appear during further search.   

Therefore, taking into account the risk (omitting the last opportunity to pick the desired object or accept the one that is not the best) and purpose (choosing the best), the objective for DM is to maximize the expected payoff. It is an attitude towards risk in the problems of choosing the best object (v. \cite{WebBlaBet2002:RiskAttitude}). The task comes down to determining the value of the problem  
\begin{align}\label{KSzPR1}
v&=\displaystyle\sup_{\tau\in\cS^N}\left[\alpha\bP\{\omega: Z_{\tau}=1,\tau\leq N\}-\beta \bP\{\omega: Z_{\tau}\neq 1,\tau\leq N\}\right.\\
\nonumber&\qquad\left.-\gamma\bP\{\omega: Z_{\tau}\neq 1,\tau\geq N\}\right].\\
\intertext{and proving the existence of some strategy $\tau^\star$}
\label{KSzPR2}
&\left[\alpha\bP\{\omega: Z_{\tau^\star}=1,\tau^\star\leq N\}-\beta \bP\{\omega: Z_{\tau^\star}\neq 1,\tau^\star\leq N\}\right.\\
\nonumber&\qquad\left.-\gamma\bP\{\omega: Z_{\tau^\star}\neq 1,\tau^\star\geq N\}\right]=v, 
\end{align}
% \st{that realizes an extreme value}
which maximizes the probability of selecting the best facility and at the same time minimizes the risk of ending the procedure without selecting or selecting an unsatisfactory facility. 
%\st{If such a strategy exists, it should be outlined.}
If such a strategy exists, it should be defined and described so that it is understandable to the average decision maker. The event $\{\omega: Z_{\tau^\star}=1,\tau^\star\leq N\}$ will be called $\textbf{DM}_\text{win}$. \label{KSzDMwin}

The problem can be reduced to the optimal stopping problem for the homogeneous Markov chain (v.\cite{DynYus1969:MarkovProcProbl}, \cite{Sza1982:a-th}) with suitable payoff functions. With the above-adopted method of defining the best object, the decision-maker should be interested in applicants with relative ranks $1$. It will be called the \emph{candidate}. We will therefore describe the moments when there is a chance for success for him. Let $W_1 = 1$. Define $ W_t = \inf\{r>W_{t-1}:Y_r=1\},\ t>1, $ ($\inf\emptyset = \partial $). $(W_t,\cF_t,\bP_{r})_{t=1}^N$, $r\in\bar{\BbbE}$, is the homogeneous Markov chain with the state space $\bar{\BbbE} = \{1,2,...,N\}\cup\{\partial\}$, $\cF_t = \sigma(W_1,W_2,...,W_t)$ and the following one-step transition probabilities: $ p(r,s) = \bP \{W_{t+1}=s\mid W_t = r\} = {\frac{r}{s(s-1)}}$ if $1\leq r<s\leq N$, $p(r,\partial) = 1 - \sum_{s=r+1}^N p(r,s)$, $p(\partial,\partial) = 1$ and $0$ otherwise. The transition probability $p(r,s)$ is the chance that next candidate appears at moment $s$ given that at current moment $r$ is the relative $1$ (i.e. $Y_r=1$). When the \textbf{DM} reaches the state $\partial$, the process automatically stops and should pay penalty $\gamma$. Let $f:\BbbE\rightarrow \Re$. The operator $\bT f(k)=\bE_k f(W_1)=\sum_{s\in\bar{\BbbE}}f(s) p(k,s)=\sum_{s=k+1}^N p(k|s)f(s)+f(\partial)p(k|\partial)$, where $\bE_\cdot f(\cdot)$ is the expected value with respect to transition probabilities. The option under consideration at the moment $k$ which is the candidate for the winning choice should be the relative first. The process is called to be in the state $k$, where $1\leq k\leq n$, if the decision maker has the $k$-th option to be a candidate. Let us assume that no stop has been made until the state $k$. When choosing this option, the expected reward is equal
\begin{equation}\label{reward1}
g(k)=\left\{\begin{array}{ll}
(\alpha+\beta)\frac{k}{N}-\beta.&\text{ if $k\in\{1,2,\ldots,N\}$}\\
-\gamma &\text {otherwise.}
\end{array}
\right.
\end{equation} 

The classical secretary problem is the case when $\alpha=1$ and $\beta=0$ (cf. \cite{Sak1984:Bilateral}, \cite{Fer2016:Odds}). The case of penalty equals $\beta=1$ is also ibid (see also \cite[sec. 5.2]{Rib2018:OnLineFirst}). The main objective to investigate is to check the player's strategy behavior for their various values. The additional effect is the observation that the optimal strategy could be derived based on the odds-theorem (cf.  \cite{BruPai2000:Selecting}, \cite{Bru2000:Odds}, \cite{Den2013:Odds}). 

Denote $V(k)$ the maximum expected payoff when the process $W_t$ is in the state $k$. The optimality principle of dynamic programming gives the equation:
\begin{align}\nonumber
V(k)&=\max\{g(k),\bT V(k)\}\\[-2ex]
\label{optrewardeq}&\quad \\[-2ex]
\nonumber&=\max\{(\alpha+\beta)\frac{k}{N}-\beta,\sum_{j=k+1}^N\frac{k}{j(j-1)}(\gamma+V(j))-\gamma\},
\end{align}
$k=1,2,\ldots,N$, $V(N)=\alpha$.

%\st{It is not advisable to emphasize that the equation \eqref{optrewardeq} is not the same as the one in the classical secretary problem because of the different reward. The solution will be different and the probability of selecting the best option will also be  different.}
It is worth emphasizing that the equation \eqref{optrewardeq} is different from the analogous one in the classic \textbf{BCP} due to the different payout function. Thus, one can expect a different optimal strategy (\emph{optimal stopping time}) and a chance to choose the best object when applying it. Let us define 
\begin{equation}\label{Hkn}
H_{k,N}=\sum_{j=k}^{N-1}\frac{1}{j},
\end{equation}
where $\frac{k}{N}H_{k,N}$ is the chance that the next candidate appears at moment $j>k$ given that at the current moment $k$ is also the candidate.
\begin{nthm}[The value and the optimal threshold.]\label{OneDMNIBCP} 
\begin{comment}
The solution of \eqref{optrewardeq} is 
\begin{equation}\label{MEPV}
V(k)=\left\{
    \begin{array}{ll}
        \frac{k^\star-1}{N}\left((\alpha+\beta)H_{k^\star-1,N}+\beta\right)-\beta&\text{for $k=1,2,\ldots,k^\star-1$}\\
        (\alpha+\beta)\frac{k}{N}-\beta&\text{for $k=k^\star,\ldots,N$,}
    \end{array}
\right.
\end{equation}
where $k^\star=k^\star(\alpha,\beta,\gamma)$ is the integer $k$ for which $H_{k-1,n}\geq\frac{\alpha}{\alpha+\beta}>H_{k,N}$.
\end{comment}
The solution of \eqref{optrewardeq} is 
\begin{equation}\label{MEPV}
V(k)=\left\{
    \begin{array}{ll}
        \frac{k^\star-1}{N}\left((\alpha+\beta)H_{k^\star,N}+\beta-\gamma\right)-\beta&\text{for $k=1,2,\ldots,k^\star-1$}\\
        (\alpha+\beta)\frac{k}{N}-\beta&\text{for $k=k^\star,\ldots,N$,}
    \end{array}
\right.
\end{equation}
where $k^\star$ is the integer $k$ for which $H_{k-1,N}\geq\frac{\alpha+\gamma}{\alpha+\beta}>H_{k,n}$. 
\end{nthm}
\begin{proof}
The proof is based on the construction of the one-step look-ahead stopping region
\begin{align*}
    B&=\{k: g(k)\geq {\mathbf E}_kg(W_1)\}\\
		&=\{k: (\alpha+\beta)\frac{k}{n}-\beta\geq \frac{k}{n}\left((\alpha+\beta)H_{k,n}+\beta-\gamma\right)-\beta\\
    &=\{k:H_{k,n}\leq \frac{\alpha+\gamma}{\alpha+\beta}\}.
\end{align*}
This is the optimal stopping region, since the conditions for the monotone case of \citeauthor{ChoRobSig1971:Great}(\citeyear{ChoRobSig1971:Great}) (see also \citeauthor{Shi2008:OSR}~(\citeyear{Shi2008:OSR})) are fulfilled. The asymptotic result is evident.\qed 
\end{proof}

\begin{table}[tbh!]
\centering
\caption{\label{KSzExDwin}The optimal thresholds and the probability for choosing the best applicant when $\alpha=1$ and some levels of penalty for each type of failure.}
\begin{tabular}{|c|c||l|l|}\hline
$\beta$&$\gamma$&\multicolumn{1}{|c|}{ $t^*(1,\beta,\gamma)$}&\multicolumn{1}{|c|}{${\mathbf P}(\textbf{DM}_\text{win})$}\\ \hline\hline 
1&0&$\exp(-.5)\cong 0.60653$ &$\frac{1}{2}\exp(-.5)\cong 0.30327$ \\[0.25ex]
0.75&0.25&$\exp(-\frac{5}{7})\cong 0.48954$ &$\frac{5}{7}\exp(-\frac{5}{7})\cong 0.34967$\\[0.25ex]
$\beta_0$&$\beta_0$&$\exp(-1)\cong 0.36788$ &\hspace{1em}$\exp(-1)\cong 0.36788$ \\[0.25ex]
0.25&0.75&$\exp(-\frac{7}{5})\cong 0.24660$&\! $\frac{7}{5}\exp(-\frac{7}{5})\cong 0.34524$\\ [0.25ex]
0&1&$\exp(-2)\cong 0.13534$                &$2 \exp(-2)\cong 0.27067$ \\ \hline
\end{tabular}
\end{table}
\begin{ncor}\label{MEKSzOSCorr}
The optimal strategy of the decision maker is $\tau^\star=\inf\{1\leq k\leq n: \text{the first stage }  k\geq k^\star-1\}$.
\end{ncor}
In other words, the optimal policy is to reject the first $k^\star-1$ options and then to accept the first candidate thereafter. 
\begin{ncor}\label{KSzAsymPR12}When $N\rightarrow\infty$ we have:
\begin{align}\label{KSzAsTres}
    t^\star=t^\star(\alpha,\beta,\gamma)&=\lim_{N\rightarrow\infty} \frac{k^\star}{N}=\exp\left(-\frac{\alpha+\gamma}{\alpha+\beta}\right)\\
\label{KSzAsVal}    V(0^{+})&=\lim_{N\rightarrow\infty} \frac{k^\star-1}{N}\left((\alpha+\beta)H_{k^\star,N}+\beta-\gamma\right)-\beta\\
		&=(\alpha+\beta)\exp\left(-\frac{\alpha+\gamma}{\alpha+\beta}\right)-\beta\\
\label{KSzAsProb} {\mathbf P}(\textbf{DM}_\text{win}) &=\frac{\alpha+\gamma}{\alpha+\beta}\exp\left(-\frac{\alpha+\gamma}{\alpha+\beta}\right).
\end{align}
\end{ncor}
The proof of asymptotic dependency is based on the approximation of difference equations by differential equations (cf.~\cite{DynYus1969:MarkovProcProbl}, \cite{Muc1973:Diff}, \cite{Muc1973:CSP}, \cite{Sza1982:a-th}). The thresholds behavior as function of $\alpha$, $\beta$ and $\gamma$ according the formulae \eqref{KSzAsTres} is presented on Figures~\ref{KSzSiec23} and \ref{fig:3DThreshold2} (i.e. optimal thresholds of the problem~\eqref{KSzPR2}). The probability of success of the optimal strategy and optimal threshold for some levels of concern about choosing the wrong option or not being able to make a selection are presented in Table~\ref{KSzExDwin}. Increasing the fear of being left without selection speeds up the acceptance of candidates, while the fear of getting the wrong option tends to delay one.

\begin{figure}[tbh!]
{\noindent
\small
\begin{subfigure}[b]{0.45\textwidth}
\centering {\includegraphics[width=\textwidth]{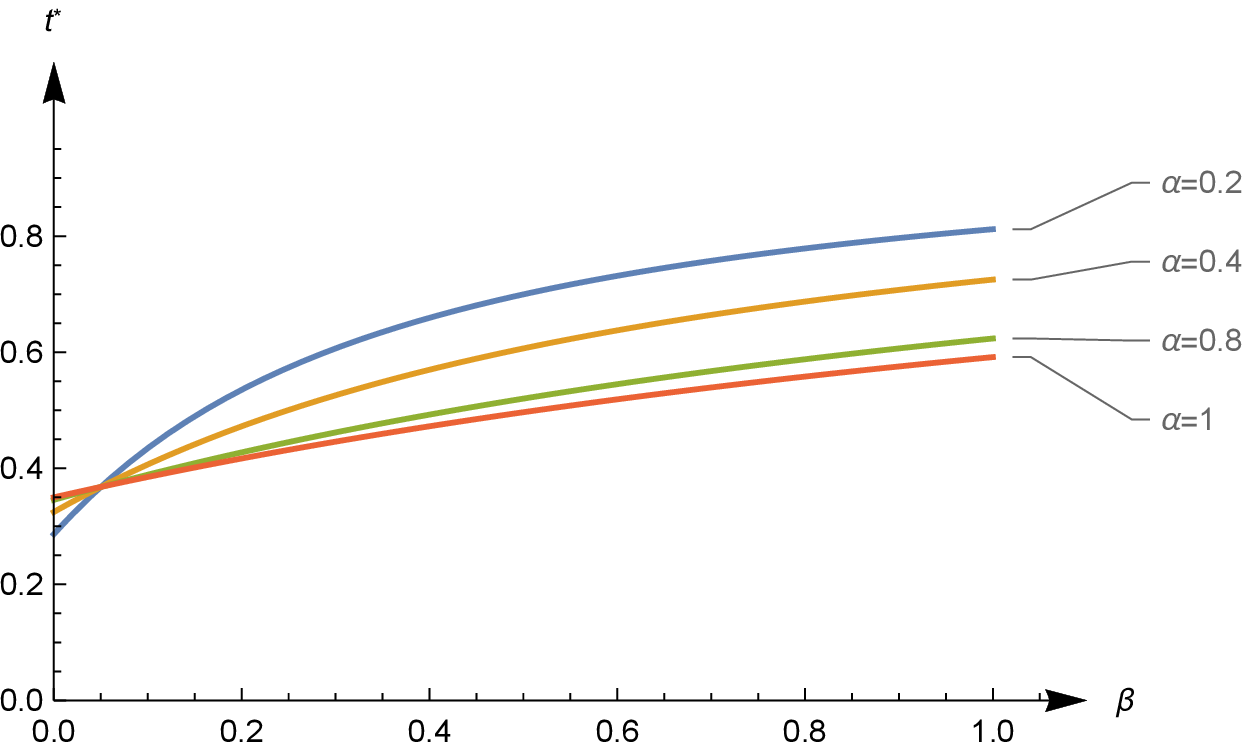}}%,height=7.5cm
\caption{\label{siec2a}$\gamma=0.05$.}%\\[3mm]
\end{subfigure}
\hfill\begin{subfigure}[b]{0.45\textwidth}
\centering {\includegraphics[width=\textwidth]{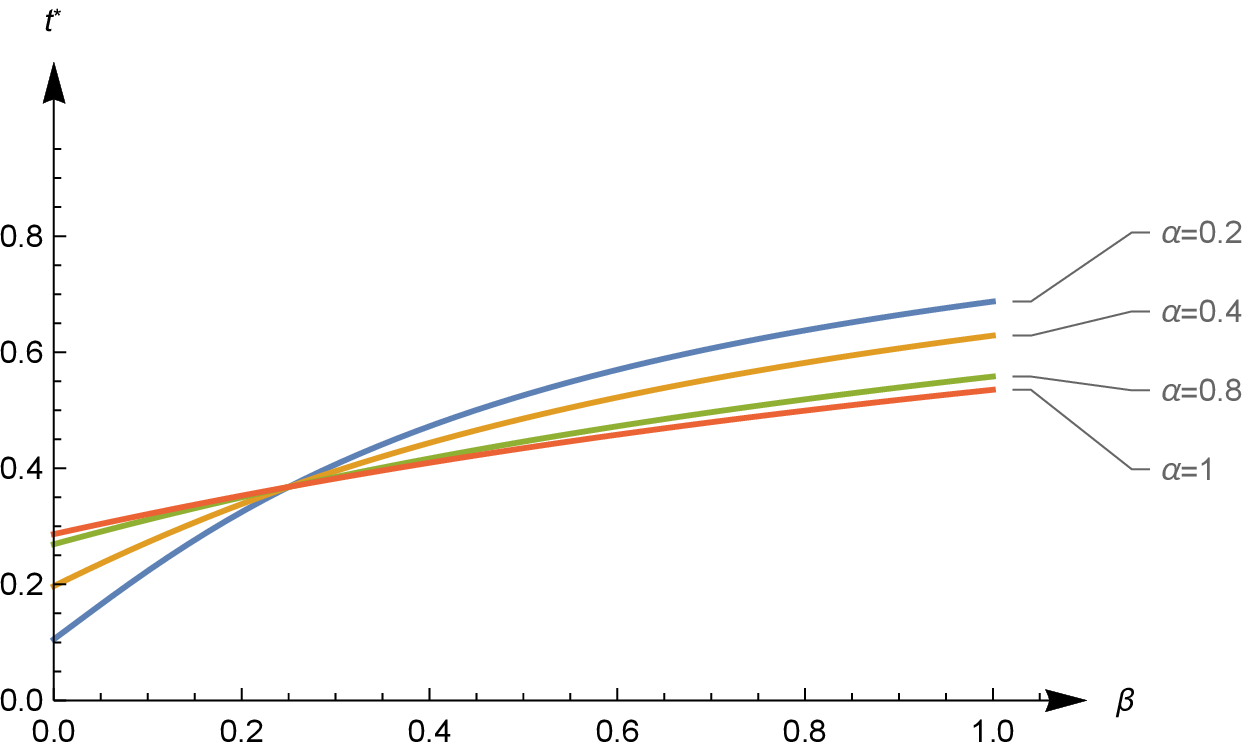}}%,height=7.5cm
\caption{\label{siec3a}$\gamma=0.25$.}%\\[3mm]
\end{subfigure}
}
\vspace{-1mm}
{\noindent
\small
\begin{subfigure}[b]{0.45\textwidth}
\centering {\includegraphics[width=\textwidth]{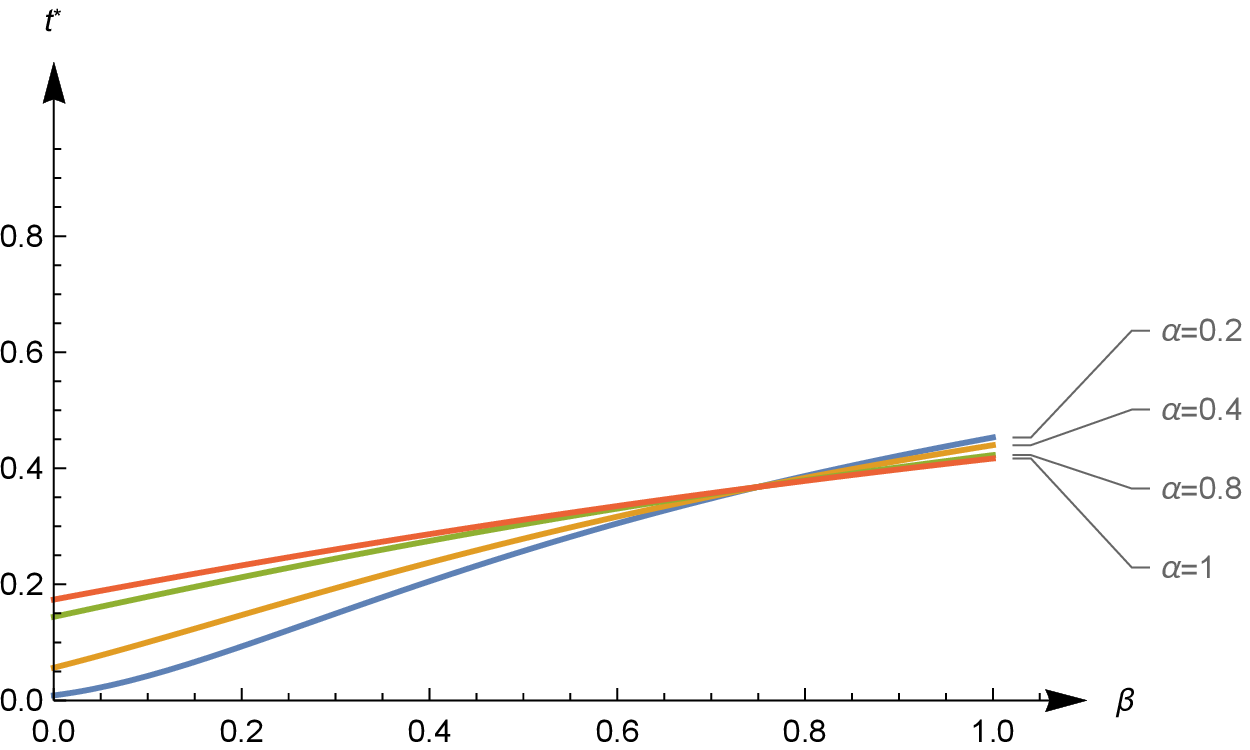}}%,height=7.5cm
\caption{\label{siec2b}$\gamma=0.75$.}%\\[3mm]
\end{subfigure}
\hfill\begin{subfigure}[b]{0.45\textwidth}
\centering {\includegraphics[width=\textwidth]{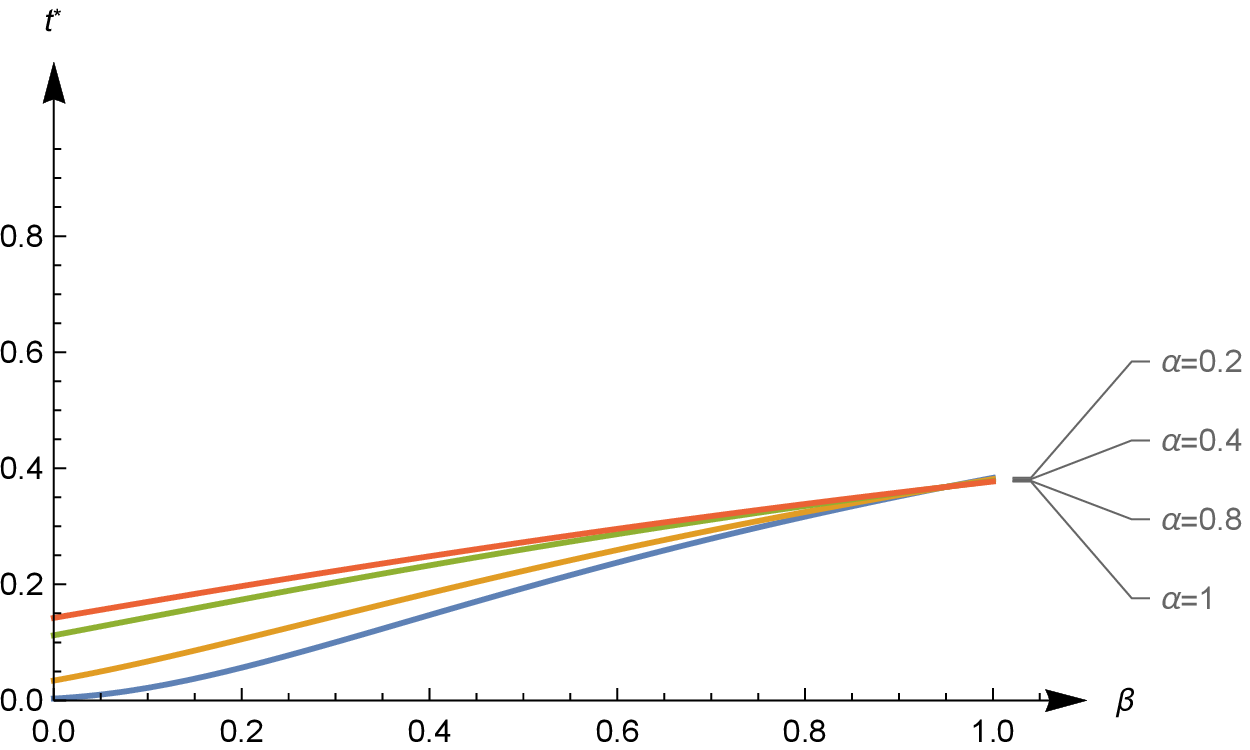}}%,height=7.5cm
\caption{\label{siec3b}$\gamma=0.95$.}%\\[3mm]
\end{subfigure}
}
\caption{\label{KSzSiec23}Thresholds behavior as function of $\alpha$, $\beta$ and $\gamma$ in the optimal Markov time of the problem~\eqref{KSzPR2}.}
%\begin{minipage}{160mm}
%\hspace{9mm} {9. Jan Ptaszycki w ogrodzie} \hspace{35mm} {10. Jan, Adam i Felicja Ptaszyccy}
%\end{minipage}
\end{figure}

\subsection{\label{MEKSzUEmp}Penalize the wrong choice when an uncertain employment is possible.} Let us consider the selection problem of Section~\ref{KSzVarRisk} and assume that the solicited candidate can refuse the proposal with probability $1-p$ (cf. \cite{Smi1975:uncertain}). Meeting with rejection, the selector has the right and obligation to further search. With each new candidate, the situation may repeat itself: the candidate to whom we propose may refuse, but there is also a risk that the candidate (i.e., new and relatively first) will not appear. Then, there is \emph{an additional loss related to the lack of candidates}. The availability of the candidate is modeled by a sequence of binary random variables. Choosing the option given availability,  the expected reward is defined by \eqref{reward1}. Denote $U(k)$ the maximum expected payoff when the process is in the state $k$. By assumption $U(\partial)=-\gamma$. The optimality principle of the dynamic programming gives the equation:
\begin{align}\label{optrewUEmp}
U(k)&=\max\big\{pg(k)+(1-p)\bT g(k),\bT U(k)\big\}\\
\nonumber&=\max\Big\{p[(\alpha+\beta)\frac{k}{N}-\beta]+(1-p)\left[\sum_{j=k+1}^N\frac{k}{j(j-1)}U(j)-\gamma\frac{k}{N}\right],\\
\nonumber&\hspace{8em}\left[\sum_{j=k+1}^N\frac{k}{j(j-1)}U(j)-\gamma\frac{k}{N}\right]\Big\},
\end{align}
$k=1,2,\ldots,N$, $U(N)=p(\alpha+\gamma)-\gamma$. The state $k$ belongs to the stopping region iff $g(k)\geq \bT U(k)$. 

Note that \eqref{optrewUEmp} covers the case of the classic problem of choosing the best object when we put $\alpha = 1$, $\beta = \gamma = 0$.
%\st{It is worst to emphasize that the equation \eqref{optrewUEmp} is not the same as the one in the classical secretary problem because of the different reward here. By this fact, the solution will be different and the probability of selecting the best option will be also different.} 
Let $v(t)=\lim_{N\rightarrow\infty}(\bT U(k)+\gamma\frac{k}{N})$ such that $\lim_{N\rightarrow\infty}\frac{k}{N}=t\in(0,1]$. On the left hand side a neighborhood of $1$ on limit of \eqref{optrewUEmp} we can observe that 
\begin{equation}\label{vEquation}
v(t)=\int_t^1\frac{t}{s^2}\left[[(\alpha+\beta)s-\beta]p+(1-p)[v(s)-\gamma s]\right]ds,
\end{equation}
with the boundary condition $v(1)=0$.
\begin{nlem}\label{SolvEquation}
The equation \eqref{vEquation} has a solution
\begin{equation*}
v(t)=\frac{\alpha p+\beta- (1-p)\gamma}{1-p}\left(t^p-t\right)+\beta t-\beta.
\end{equation*}
\end{nlem} 
The function $g(t)$ is increasing and it has one intersection with $v(t)-\gamma t$. The solution $t^\star$ of the equation $g(t)=v(t)$ in $(0,1]$ is the optimal threshold for the strategy. 
%\begin{proof}
The equation
\begin{align*}
(\alpha+\beta)t-\beta&=\frac{\alpha p+\beta- (1-p)\gamma}{1-p}\left(t^p-t\right)+(\beta-\gamma) t-\beta\\
\intertext{can be reformulated to}
\frac{p\alpha +\beta-(1-p)\gamma}{1-p}t^p&=\frac{\alpha+\beta}{1-p}t.  
\intertext{This implies} 
t^\star&=\left(\frac{p\alpha +\beta-(1-p)\gamma}{\alpha+\beta}\right)^{\frac{1}{1-p}}.
\end{align*}%\end{proof}
\begin{figure}[tbh!]
		\centering
		\includegraphics[width=10.5cm]{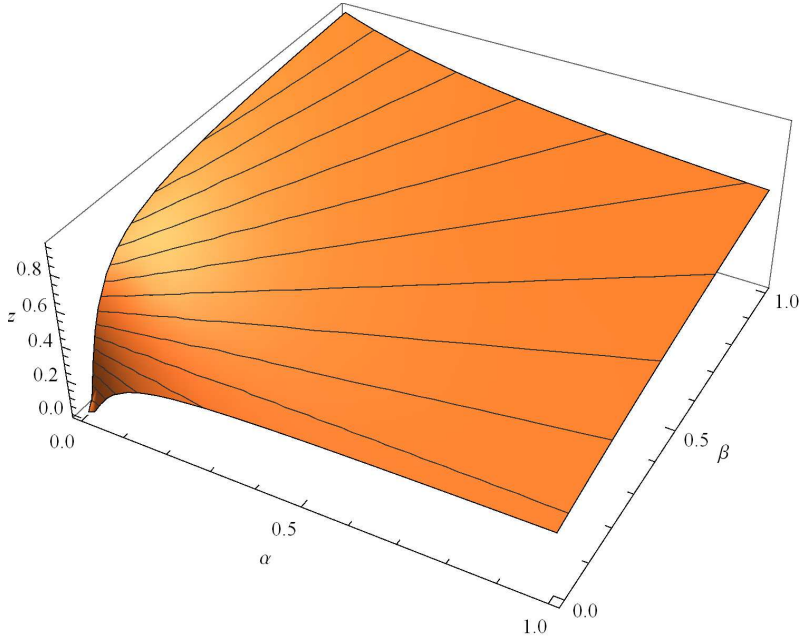} 
		\caption{\label{fig:3DThreshold2}Thresholds as function of $\alpha$ and $\beta$ for $\gamma=.05$ and $p=.95$.} 
	\end{figure}
\begin{nthm}[BCP with uncertain emplyment.]\label{OneUnBCP} The asymptotic solution of \eqref{optrewUEmp} is 
\begin{align}\label{MEPW}
u(t)&=\left\{
    \begin{array}{ll}
        p[(\alpha+\beta)t-\beta]+(1-p)[v(t)-\gamma t] &\text{ for $t\geq t^\star$,}\\
        (\alpha+\beta)t^\star-\beta &\text{ for $t< t^{\star}$.}
    \end{array}
\right.\\
\nonumber &=\left\{
    \begin{array}{ll}
        \frac{p(\alpha+\beta)+(1-p)(\beta-\gamma)}{1-p}(t^p-t)+(\beta-\gamma) t-\beta &\text{ for $t\geq t^\star$,}\\
        (\alpha+\beta)t^\star-\beta &\text{ for $t< t^{\star}$.}
    \end{array}
\right.
\end{align}
\end{nthm}
\begin{proof}
The proof is based on the limit in the construction of the one-step look ahead stopping region\footnote{Stopping region when the current observation is better than the average of the next one.}. Based on the functional equation~\eqref{optrewUEmp} the function $\hat{u}(t)=\lim_{N\rightarrow\infty}\bT U(k)$ should fulfill the integral equation
\begin{equation}\label{EMKSzInEq}
\hat{u}(t)=\int_t^1\frac{t}{s^2}\left[p((\alpha+\beta)s-\beta)+(1-p)(\hat{u}(s)-\gamma s)\right]ds,
\end{equation}
with the boundary condition $\hat{u}(1)=\gamma$. The equivalent differential equation has the form
\begin{equation}\label{EMKSzDiffEq}
\hat{u}'(t)=p\frac{\hat{u}(t)}{t}+p\frac{\beta}{t}-(p(\alpha+\beta)-(1-p)\gamma)\text{ with $\hat{u}'(1)=-\gamma$}.
\end{equation}
The solution of the equation~\eqref{EMKSzDiffEq} 
\begin{equation}\label{EMKSzDiffEqSol}
\hat{u}(t)=\frac{p(\alpha+\beta)+(1-p)(\beta-\gamma)}{1-p}(t^p-t)+\beta t-\beta.
\end{equation}
The state $k\approx Nt$ on the border of the one-step look ahead region fulfills the equation $(\alpha+\beta)t-\beta=\hat{u}(t)-\gamma t$ and, by \eqref{EMKSzDiffEq}, $\hat{u}(t)=\frac{p\alpha+\beta-(1-p)\gamma}{1-p}(t^p-t)+\beta t-\beta$. It implies the asymptotic stopping region 
\begin{align*}
    B&=\{t\in(0,1]: t\geq  \left(\frac{p(\alpha+\beta)+(1-p)(\beta-\gamma)}{\alpha+\beta}\right)^{\frac{1}{1-p}}\}.
\end{align*}
This is the optimal stopping region, since the conditions for the monotone case of \cite{ChoRobSig1971:Great} are fulfilled. These conclude the thesis.\qed
\end{proof}

\begin{ncor}\label{KSzLemUnThres}
The optimal strategy for the best choice problem with uncertain employment and the penalties is to accept  the relatively best until the chosen one accepts the proposal, starting from the moment \begin{equation}\label{KSzUnThres}
t^\star(\alpha,\beta,\gamma,p)=\left(\frac{p\alpha +\beta-(1-p)\gamma}{\alpha+\beta}\right)^{\frac{1}{1-p}}.
\end{equation}
\end{ncor}
Let us observe that $\lim_{p\rightarrow 1} t^\star(\alpha,\beta,\gamma,p)=\exp\left(-\frac{\alpha+\gamma}{\alpha+\beta}\right)$ as stated in Corollary~\eqref{KSzAsymPR12} (v. formulae~\ref{KSzAsTres}).
\begin{ncor}\label{KSzp1}
The asymptotic optimal strategy of the decision maker in the case of $\beta=\gamma$ is $\tau^\star=\inf\{1\leq k\leq N: \text{The first stage } k\geq k^\star-1$\}, where $k^\star\cong [Np^{\frac{1}{1-p}}]+1$. Note that the stopping area does not depend on $\alpha$ and $\beta$! 
\end{ncor}

The observation formulated in Corollary~\ref{KSzp1} is surprising. In the problem of choosing the best object, if there is additional uncertainty related to the availability of options (random factor), which the decision-maker is not able to control, his level of interest in winning and losing no longer matters. With $p \rightarrow 1$, i.e. when the probability of availability of the analyzed options is close to $1$, both the threshold and the probability value of selecting the best object are equal to $\exp\{-1\}$, the solution of the classical no-information \textbf{BCP}.
%%% Statistical properties of the solution in sequential search problem

%%%%%%%%%%%%%%%%%%%%%%%%%%%%%%%%%%%%%%%%%%
\section{\label{KSzDurNIBCP}Expected duration of sequential decision problems.}%
\subsection{\label{DurOneStop}Duration for the one stop with the penalty cost.} 
When we start the task of sequential search for the best opportunity (possibilities, options), the problem is reformulated to the classic optimal stop model. The main objective is to optimize the criterion on a class of strategies, but there are other objectives. We do not know, \emph{a priori}, when this optimization process will be completed and how long it will take. If the cost of observation is not included in the objective function, there is no reason why the desired decision moment should be minimal or maximal. It is known from the general theory of martingales that if the withdrawal process is martingale, then each stopping time is optimal in the optimal stopping problem. For any payoff process, determining the average time spent in the decision-making process is important information for the decision maker. Similarly, in tasks where the number of permissible decision moments is greater (such as with the uncertain employment model), we are interested in the time until all options are exhausted or the possibilities for further searching are exhausted. When counting the time spent implementing the optimal strategy, its effectiveness is not analyzed. It is irrelevant in this case whether there is a wanted state among the stopped states. A similar study can be found in the article by \cite{SteSeaRap2003:HeuristicBCP}. \cite{MazPes2003:asymptotic} and recently by \cite{Dem2021:Expected}, \cite{ErnSza2021:Average} and \cite{Liu2021:QueryBasedSO}. The comparison of the average time of the decision-making process that we present contains new, interesting observations. 

\subsection{\label{KSzMTPenBCP}The mean time value of decision work with penalty for errornous choice.} The optimal strategy in this case is constructed in Section~\ref{KSzVarRisk} and is given explicitly in Theorem~\ref{OneDMNIBCP} and Corollary~\ref{MEKSzOSCorr}. The optimal strategy is the threshold one. For $\tau^\star$ given in the Corollary~\ref{MEKSzOSCorr} we have   $\bP(\omega:\tau^\star\leq N)=\frac{k^\star-1}{N}H_{k^\star,N}\cong \frac{\alpha+\gamma}{\alpha+\beta}\exp\{-\frac{\alpha+\gamma}{\alpha+\beta}\}<1$ (v. Formula \eqref{KSzAsProb}). To describe the decision-making time, we will modify this optimal Markov moment and create $ \tilde{\tau}^\star $ which are equal to the optimal Markov moments in the set $ \{\omega: \tau^\star<N \} $, and on the complement of this event they are equal to $ N $. We can interpret $ \tilde{\tau}^\star $ as the time to stop in one state or reach the last observation at $ N $. The entire decision-making process ends at $ \tilde{\tau}^\star $ and we are interested in the mean of this random variable. Let us denote $m(k)=\bE(\tilde{\tau}^\star|\cF_k)$, $k\leq N$.  
\begin{align*}%\label{m2k}
m(k)&=\left\{
\begin{array}{ll}
\displaystyle \sum_{j=k+1}^N\frac{k}{j(j-1)}j+N(1-\sum_{j=k+1}^N\frac{k}{j(j-1)}) & \mbox{ if $ k>k^\star$,}\\
\displaystyle \sum_{j=k^\star+1}^N\frac{k^\star}{j(j-1)}j+N(1-\sum_{j=k^\star+1}^N\frac{k^\star}{j(j-1)})   & \mbox{ if $ k\leq k^\star$.}
\end{array}
\right.\\
&=\left\{
\begin{array}{ll}
k H_{k,N}+k& \mbox{ if $ k>k^\star$,}\\
k^\star H_{k^\star,N}+k^\star   & \mbox{ if $ k\leq k^\star$.}
\end{array}
\right.
\end{align*}
Therefore, we can obtain the expected value of residence time in the decision-making process.  
\begin{ncor}When $n\rightarrow\infty$ we have:
\begin{align*}
    \hat{m}(t)&=\lim_{N\rightarrow\infty} \frac{m(k)}{N}=\left\{
\begin{array}{ll}
-t\ln(t)+t& \mbox{ if $ t>t^\star$,}\\
-t^\star\ln(t^\star)+t^\star   & \mbox{ if $ t\leq t^\star$.}
\end{array}
\right.\\
    m(0^+)&=[1+\frac{\alpha+\gamma}{\alpha+\beta}]\exp\left(-\frac{\alpha+\gamma}{\alpha+\beta}\right).
\end{align*}
\end{ncor}
For $\alpha=1$ and $\beta=\gamma$ we have the result of \cite{Yeo1997:Duration}.

\subsection{The mean processing time of hesitating candidates.} The optimal strategy in this case is constructed in Section~\ref{MEKSzUEmp}, and is given explicitly in Theorem~\ref{OneUnBCP} and Corollary~\ref{KSzLemUnThres}. 
\begin{align*}*\label{m2k}
m(k)&=\left\{
\begin{array}{ll}
\displaystyle \sum_{j=k+1}^N\frac{k}{j(j-1)}[pj+(1-p)m(j)]&\\
\hspace{8em}+N(1-\sum_{j=k+1}^N\frac{k}{j(j-1)}) & \mbox{ if $ k>k^\star$,}\\
\displaystyle \sum_{j=k^\star+1}^N\frac{k^\star}{j(j-1)}[pj+(1-p)m(j)]&\\
\hspace{8em}+N(1-\sum_{j=k^\star+1}^N\frac{k^\star}{j(j-1)})   & \mbox{ if $ k\leq k^\star$.}
\end{array}
\right.
\end{align*}

Asymptotic of the mean time when $N\rightarrow\infty$ has the form:
\begin{align*}
    \hat{m}(t)&=\lim_{N\rightarrow\infty} \frac{m(k)}{N}=\left\{
\begin{array}{ll}
-pt\ln(t)+t+(1-p)\displaystyle\int_t^1\frac{t}{s^2}\hat{m}(s)ds& \mbox{ if $ t>t^\star$,}\\
-pt^\star\ln(t^\star)+t^\star+(1-p)\displaystyle\int_{t^{\star}}^1\frac{t}{s^2}\hat{m}(s)ds   & \mbox{ if $ t\leq t^\star$.}
\end{array}
\right.\\
\intertext{with the boundary condition $m(1)=1$.}
\end{align*}
As in Section~\ref{KSzMTPenBCP}, it allows us to calculate the average time to complete the decision-making process. 
\begin{nlem}
The solution of the integral equation with $\hat{m}(1)=1$ has the following form:
\begin{align*}
    \hat{m}(t)&=\left\{
\begin{array}{ll}
-pt\ln(t)+t+(1-p)\int_t^1\frac{t}{s^2}\hat{m}(s)ds& \mbox{ if $ t>t^\star$,}\\
-pt^\star\ln(t^\star)+t^\star+(1-p)\int_{t^{\star}}^1\frac{t}{s^2}\hat{m}(s)ds   & \mbox{ if $ t\leq t^\star$.}
\end{array}
\right.\\
    \hat{m}(t)&=\left\{
\begin{array}{ll}
\frac{1}{1-p}(t^p-pt)& \mbox{ if $ t>t^\star$,}\\
\frac{1}{1-p}({t^\star}^p-pt^\star)   & \mbox{ if $ t\leq t^\star$.}
\end{array}
\right.\\
\intertext{with $t^\star$ depending on the parameters $\alpha,\beta,\gamma$ and $p$ (v. Formula~\eqref{KSzUnThres})}
     t^\star&=\left(\frac{p\alpha +\beta-(1-p)\gamma}{\alpha+\beta}\right)^{\frac{1}{1-p}}.
\end{align*}
\end{nlem} 
In conclusion, the asymptotic average time to complete the decision-making process with hesitating candidates 
\begin{equation*}
\hat{m}(0^+)=\frac{1}{1-p}\left[{\left(\frac{p\alpha +\beta-(1-p)\gamma}{\alpha+\beta}\right)}^{\frac{p}{1-p}}-p\left(\frac{p\alpha +\beta-(1-p)\gamma}{\alpha+\beta}\right)^{\frac{1}{1-p}}\right].
\end{equation*}    

\section{\label{KSzFinRem}Final remarks.}
The significance of the presented analyzes lies in linking the final effects of the strategy used with the visible effects for the decision maker. These, in turn, can affect the way you choose your behavior strategy. Behavioral strategies are often a consequence of habits rather than their rationalization (v. \cite{Aum2019:NHB}).

Observation of the impact of the strategy on the implementation of selected events is possible by determining the value of these events. Such parameterization of the model also allows the analysis of other properties of the solution: the value of the problem and the optimal strategy. An important feature of the solution is its implementation time, which can be used to indicate an additional desired feature of the solution, the objective function being not taken into account. Such a solution study plays a special role in the analysis of stopping games, where the number of alternative solutions is significant and players cannot communicate during the game. The author is convinced that the results obtained will allow a new approach to the analysis of behavior when making decisions on sequential issues.

In real problems, the mathematical model, which comes down to the choice of the moment, we also have a psychological aspect to be included in the model. Achieving the goal is usually a priority. However, \emph{a failure  can also be perceived differently}, and taking it into account has an impact on the choice or the option to solve the problem. Based on the analysis presented, we can see that measurable benefits and losses can significantly affect the results achieved, optimal behavior, and rational behavior. As already recalled in Section~\ref{KSzComBCP}, this phenomenon is often raised in experimental and mathematical psychology research. %(Rappaport, Bardin, Szajowski). 
Recently, this aspect has not been overlooked in the basic research of mathematical models (v. \cite{Liu2021:QueryBasedSO,CrewsEtAl2019:Opportunity}). They considered a model for finding the best candidate in the selection process, in which each attempt may end up asking the (Geni) prophet if the candidate indicated is right or not (cf. also \cite{DotMar2014:hint,Ska2019:hint}). In the latter case, the search is resumed. The order in which the candidates are presented to the interview is random. The draw follows the \cite{Mall1957:NNRank} distribution. Further research is planned for more subtle sequential searches, such as, e.g., in the work by \cite{Sak1978:UncertSP},  \cite{Sza1982:a-th}, \cite{Ros1982:nonextremal} and \cite{Van2021:postdoc}. The author hopes that the observations presented in the article, due to the possibility of parameterize the behavior of the person conducting the trial, will be used for further empirical research and will allow the construction of diagnostic tools similar to those used by \cite{graHor2020:Effort} and those used in psychiatric diagnostics (cf. \cite{SchJarBriSza2003:Inter}) 

\section*{Acknowledgement} %I would like to express my very great appreciation to Professor Francis Tuerlinckx (Associate Editor) and the anonymous Referee  for his valuable and constructive questions, comments and suggestions. 

The research has been supported by Wrocław University of Science and Technology, Faculty of Pure and Applied Mathematics,  under the project 8211204601 MPK: 9130740000.

%\begin{comment}
\renewcommand{\appendixname}{}
\appendix

\section{\label{sec:sample:appendix}Appendix}
\subsection{\label{KSzSecOptSP}The optimal stopping problems.} Let $(X_k,{\cF}_k,\bP_x)_{k=0}^N $ be a homogeneous Markov chain defined in a probability space $(\Omega,{\cF},\bP)$ with a state space $(\BbbE,\cB)$, and let $f:\BbbN\times\BbbE\rightarrow\Re$ be a sequence of real-valued $\cB$ measurable functions. Horizon $N$ is finite. The Markov chain is observed and the ``best realizations'' according to the function $f(\cdot,\cdot)$ are selected. Basically, the states of the process are related to the observations of random variables. We will also allow the possibility that there is $k$ such that $\bP_x\{X_k \in \BbbE\} \leq 1$. For consistency in notation, we will introduce the extended set of states $\bar{\BbbE} = \BbbE \cup \{\partial\}$. 
\begin{ndef}\label{MEKSzST}
The random variable $\tau:\Omega\rightarrow \BbbN$ is called a Markov moment if $\{\omega:\tau(\omega)=n\}\in\cF_n$ for every $n\in \BbbN$. The stopping time is the finite Markov moment $\tau$ \textit{a.e.}. 
\end{ndef}

Following \cite{Shi2008:OSR} (cf.~\cite{ChoRobSig1971:Great}, \cite{PesShi2006:Free-boundary}), define ${\cS^N}$ the set of Markov times with respect to $({\cF}_n)_{n=0}^N$. We admit that $\bP_x(\tau\leq N)<1$ for $\tau\in{\cS^N}$ (i.e. there is a positive probability that the Markov chain will not be stopped). The elements of ${\cS^N}$ are possible strategies in the optimal stopping problem. Denote $\cS_n^N=\{\tau\in\cS^N:\tau\geq n\}$. Let 
\begin{equation}\label{MEKSzOSV}
v(n,x)=\sum_{\tau\in\cS_n^N}\bE[f(\tau,X_{\tau})|X_n=x].
\end{equation}
The function $v(0,x)$ is the value of the optimal stopping problem. 

\begin{ndef}\label{MEKSzOSP}
The optimal stopping problem for Markov sequences consists of determining the value function $v(0,x)$ and determining the Markov time $\tau^\star\in\cS$ such that $v(0,x)=\bE_x f(X_{\tau^\star})$ if it exists. If such Markov time exists, it is called an optimal Markov moment (an optimal stopping time if it is finite). 
\end{ndef}
Let us be definite
\begin{align}\label{MEKSzOperT}
\bT f(n,x)&=\bE[f(n+1,X_{n+1})|X_n=x] \text{ (the mean operator);}\\
\bQ f(n,x)&=\max\{f(n,x),\bT f(n,x)\} \text{ (the maximum operator)}.\label{MEKSzOperQ}
\end{align}
\begin{nthm}
The value function $v(k,x)$ fulfills the equation.
\begin{align*}
v(N,x)&=f(N,x)\\
v(n,x)&=\max\{f(n,x),Tv(n,x)\} \text{ for $n\in\BbbN$}.
\end{align*}
The Markov time $\tau^\star=\inf\{n\in\BbbN: X_n\in A_n\}$, where $A_n=\{x\in\cB:f(n,x)\geq \bT v(n,x)\}$, is optimal.
\end{nthm}
\begin{nrem}\label{MEKSzOLA}[\text{An one-step look-a-head policy (OLA)}]
Let $B_n=\{x\in\cB:f(n,x)\geq \bT f(n,x)\}$ and $\tau_{\textbf{OLA}}=\inf\{n\in\BbbN:  X_n\in B_n\}$. The strategy $\tau_{\textbf{OLA}}$ is called \text{One-step Look-a-head Policy} (\textbf{ OLA }).
\end{nrem}
\abbreviations{The following abbreviations are used in this article:\\

\noindent 
\scriptsize\raggedleft
%\begin{leftline}
\begin{tabular}{@{}ll}
%$\One_A(t)$& a characteristic function of the set $A$ (p. \pageref{CharA})\\
$(\Omega,{\cF},\bP)$& the probability space (p. \pageref{KSzSecOptSP})\\
$Z_k$, ($Y_k$)& the absolute (relative) rank (p. \pageref{KSzRR1})\\
 \textbf{BCP}&  \emph{Best Choice Problem} (p. \pageref{KSzBCP})\\
\textbf{DM} ($\textbf{DM}_\text{win}$)& Decision-maker (\emph{the aim of the} \textbf{DM})(p. \pageref{KSzDMwin})\\
$(X_n,{\cF}_n,\bP_x)$&the Markov process (p. \pageref{KSzSecOptSP})\\ 
${\cF}_n$& the filtration (p. \pageref{KSzSecOptSP})\\
$\bT$ & the mean operator of the Markov system (p. \pageref{MEKSzOperT})\\
$\bQ$ & the maximum operator (p. \pageref{MEKSzOperQ})\\
$v(k,x)$& the value of the stopping problem (p. \pageref{MEKSzOSP})\\
$\tau$, $\sigma$& Markov moments, stopping times (p. \pageref{MEKSzST})\\
\textbf{OLA} policy& the one step look ahead strategy (p. \pageref{MEKSzOLA}).
\end{tabular}
%\end{leftline}
}
\normalsize
\medskip
%%\input{BCPwithPenalty2JMP16v22bbl}

%\end{comment}
%%======================================= beginning bibliography with bib =================================================
%\input{Parts/MA49MEKSz14vi21Durationbib}
%\bibliographystyle{abbrvnat}%elsarticle-harv} 
%\bibliography{\jobname}
%======================================= end bibliography with bib =================================================





% \end{thebibliography}
\end{document}